\documentclass[12pt,a4paper,reqno]{amsart}

\usepackage{amsmath,mathrsfs}
\usepackage{amssymb}
\usepackage{color}
\usepackage{dsfont}
\usepackage[latin1]{inputenc}

\newcommand\NoBlackBoxes{\global\overfullrule0pt}
\NoBlackBoxes

\usepackage{todonotes}
\allowdisplaybreaks

\makeatletter
\let\serieslogo@\relax
\let\@setcopyright\relax

\newtheorem{definition}{Definition}[section]
\newtheorem{theorem}[definition]{Theorem}

\newtheorem{rem}[definition]{Remark}

\newtheorem{corollary}[definition]{Corollary}

\newcommand{\R}{{\mathbb{R}}}

\renewcommand{\epsilon}{\varepsilon}
\renewcommand{\phi}{\varphi}

\newcommand{\be}{\begin{equation}}
\newcommand{\ee}{\end{equation}}
\begin{document}

\title[Associative memories for sparse patterns with large capacity]{On associative neural networks for sparse patterns with large capacities}

\author[Matthias L\"owe]{Matthias L\"owe}
\address[Matthias L\"owe]{Fachbereich Mathematik und Informatik,
University of M\"unster,
Einsteinstra\ss e 62,
48149 M\"unster,
Germany}

\email[Matthias L\"owe]{maloewe@math.uni-muenster.de}

\author[Franck Vermet]{Franck Vermet }
\address[Franck Vermet]{Univ Brest, CNRS UMR 6205, Laboratoire de Math\'ematiques de Bretagne Atlantique,  6, avenue Victor Le Gorgeu\\
CS 93837\\
F-29238 BREST Cedex 3\\
France}

\email[Franck Vermet]{Franck.Vermet@univ-brest.fr}

\thanks{Corresponding author : Franck Vermet. Email: \texttt{Franck.Vermet@univ-brest.fr}}


\date{\today}



\begin{abstract}

Generalized Hopfield models with higher-order or exponential interaction terms
are known to have substantially larger storage capacities than the classical
quadratic model. On the other hand, associative memories for sparse patterns,
such as the Willshaw and Amari models, already exhibit enhanced storage
capacities in the sparse regime.

In this paper we combine these two mechanisms. We introduce higher-order
versions of sparse associative memory models and study their storage capacities 
in the sense of fixed-pattern  stability. For the Amari and Willshaw
models with fixed interaction order $n$, we obtain storage scales of order
$\frac{N^n}{(\log N)^n}$.
When the interaction order grows logarithmically with the number of neurons,
the resulting storage scale becomes super-polynomial. We also study
higher-order interactions in the block-structured Gripon--Berrou architecture,
where the natural storage scale is of order $c^n$.

Our results show that the capacity increase caused by higher-order interactions
persists in the sparse setting, while the precise storage scale depends on the
underlying architecture.
\end{abstract}

\maketitle



\section{Introduction}

Since the seminal paper of Hopfield \cite{Hopfield1982}, neural networks with
Hebbian interactions have been studied as models of associative memory. In the
classical Hopfield model, one stores $M$ random patterns
\[
\xi^1,\dots,\xi^M\in\{-1,+1\}^N
\]
by means of quadratic interactions, and one asks whether the resulting dynamics
is able to recover these patterns from the induced energy landscape. Depending on
the notion of retrieval, the model stores at most linearly many patterns in the
number $N$ of neurons; see, for instance,
\cite{MPRV, burshtein, Newman_hopfield, loukianova, Talagrand_hopfield}.
Extensions to correlated patterns and diluted architectures were studied in
\cite{Lo98, LV05, LV11, LV14}.

When the stored patterns are sparse, i.e.\ when the active state occurs only
rarely, the classical Hopfield model is no longer the most natural associative
memory model. In that regime, several alternatives have been proposed, including
the models of Willshaw \cite{Willshaw}, Amari \cite{Amari1989}, and related
systems \cite{Okada1996, LV_BEG, griponc}. Their performance was compared systematically in \cite{GHLV16}. A general
conclusion is that sparsity alone can lead to a substantial increase in storage
capacity, with the precise scale depending strongly on the architecture and on
the notion of retrieval.

A different mechanism for increasing the capacity of associative memories is to
replace the quadratic interaction term by higher-order interactions. Variants of
this type were considered already in \cite{Newman_hopfield} and were recently
revisited by Krotov and Hopfield \cite{KrotovHopfield2016}. For interaction
polynomials of degree $n$, the storage capacity increases to polynomial order in
$N$, more precisely to scales of order $N^{n-1}$; rigorous results in this
direction were obtained in \cite{Newman_hopfield, DHLUV17}. Moreover, in
\cite{DHLUV17} it was shown that, in a suitable large-order limit corresponding
to exponential interactions, one obtains exponentially large storage capacities. 
A continuous version of this model with exponential storage capacity was
proposed in \cite{ramsauer2020hopfield}, where a close relation to the attention
mechanism used in Transformers was established. More recently, connections
between associative memories and diffusion-based generative models have also
been explored in \cite{ambrogioni2024search,pham2025memorization}. These
developments illustrate that questions of memorization, retrieval, and
information storage in associative memory models continue to be relevant in
modern machine learning. A variant of the exponential Hopfield model, again
featuring exponential storage capacities, was recently proposed in
\cite{albanese2025exponentialhopfieldmodel}. Questions of biological
plausibility have likewise motivated numerous variants of Hopfield-type models;
see, for example, \cite{griponc,Agliari_2020,krotov2021largeassociativememoryproblem}.

The aim of the present paper is to  study associative
memory models for sparse binary patterns with higher-order interaction terms,
thereby linking the sparse models investigated in \cite{GHLV16} with the huge
capacity mechanism of \cite{DHLUV17}. In particular, we introduce higher-order versions of the Amari and Willshaw
models and investigate the resulting storage capacities. For fixed interaction
order $n$, their natural storage scale is of order
$\frac{N^n}{(\log N)^n}$.
If the interaction order grows logarithmically with the system size, the
resulting scale becomes super-polynomial. We also study the corresponding
higher-order construction in the Gripon--Berrou architecture (see \cite{griponc}), where the block
structure leads instead to the natural scale $M\asymp c^n$.

Our results are fixed-point stability results: for any fixed stored pattern,
the probability that it is stable under the corresponding one-step dynamics
tends to one. We do not claim simultaneous stability of all stored patterns,
nor do we address here the size of basins of attraction or the correction of
partially corrupted patterns. Nevertheless, the results show that the capacity increase due to
higher-order interactions is robust and persists far beyond the classical dense
Hopfield setting.

The remainder of the paper is organized as follows. In the next section we define
the higher-order sparse associative memory models studied here. The main results
on storage capacities are stated in Section~3, where we also discuss the effect
of growing interaction order and the corresponding Gripon--Berrou variant. The
proofs are given in the final section.

\section{The models}

Let us first describe the kind of patterns we want to store in our associative memories.

For the Amari and Willshaw models considered below, we take independent
random patterns
\[
\xi^1,\ldots,\xi^M\in\{0,1\}^N,
\]
whose coordinates are i.i.d.\ Bernoulli random variables with
\[
\mathbb P(\xi_i^\mu=1)=p_N:=\frac{\log N}{N}.
\]
Thus a pattern contains on average $\log N$ active neurons, corresponding to
a very sparse regime.

Notice that, for every fixed $\mu$,
\[
\mathbb P(\xi^\mu\equiv0)=(1-p_N)^N\sim\frac1N,
\]
and hence a fixed stored pattern is nonzero with probability tending to one.
If one considers substantially sparser regimes in which $Np_N$ does not tend
to infinity, even a fixed pattern may have a nonvanishing probability of being
identically zero, leading to an obvious degeneracy.

The Gripon--Berrou model considered below uses a different, block-structured
pattern ensemble, which will be defined separately.

All retrieval maps in the paper are applied synchronously: that is, given a
configuration $\sigma$, all coordinates of the updated configuration are
computed from the same input $\sigma$.

\subsection{Amari's models}

The first class of models we consider is inspired by work of Amari
\cite{Amari1989}. It is the sparse analogue of a generalized Hopfield model with
higher-order interactions.

Fix an integer interaction order $n\geq2$. For a configuration
$\sigma=(\sigma_i)_{i=1}^N\in\{0,1\}^N$, define the retrieval dynamics
$T=(T_i)_{i=1}^N$ by
\begin{equation}\label{multdyn}
T_i(\sigma):=
\Theta\left(\sum_{1\le j_1,\ldots,j_{n-1}\le N}^{*}
\sigma_{j_1}\cdots \sigma_{j_{n-1}}W_{i,j_1,\ldots,j_{n-1}}-h\right),
\end{equation}
where the restricted sum $\sum^*$ runs over indices
$j_1,\ldots,j_{n-1}$ that are pairwise distinct and different from $i$ and
\begin{equation}\label{multhebb}
W_{i_1,\ldots,i_n}
:=\sum_{\mu=1}^M\xi_{i_1}^\mu \xi_{i_2}^\mu \cdots \xi_{i_n}^\mu.
\end{equation}
Here $\Theta(x)=\mathbf 1_{\{x>0\}}$ denotes the Heaviside function. We choose
\[
h=\gamma \log^{n-1} N
\]
for some $\gamma>0$. A threshold is necessary here, since otherwise the local fields would typically force all neurons to become active.

Given an input configuration $\sigma$, the resulting updated configuration is
\[
T(\sigma):=(T_1(\sigma),\ldots,T_N(\sigma)).
\]
Thus a stored pattern $\xi^\mu$ is a fixed point if
$T(\xi^\mu)=\xi^\mu.$

Next we consider a logarithmically growing interaction order. For fixed $n$, the
storage capacity remains of polynomial order in $N$. To obtain a genuinely larger
capacity while staying within the sparse regime, it is natural to let the
interaction order grow with the system size. Since the typical number of active
neurons in a pattern is of order $\log N$, the choice
\[
n=\kappa \log N
\]
with $0<\kappa<1$ is the natural sparse analogue of the large-order regime. While
in \cite{DHLUV17} this limit was realized through an exponential interaction
function, here we implement it directly by letting $n$ depend on $N$. This leads to the dynamics $\widehat T=(\widehat T_i)_{i=1}^N$
defined by
\begin{equation}\label{multdyn_inf}
\widehat T_i(\sigma):=
\Theta\left(
\sum_{1\le j_1,\ldots,j_{n-1}\le N}^{*}
\sigma_{j_1}\cdots \sigma_{j_{n-1}}
W_{i,j_1,\ldots,j_{n-1}}
-h
\right),
\end{equation}
where now $n=\kappa\log N$ and the restricted sum $\sum^*$ has the same meaning as in \eqref{multdyn}.
For notational simplicity, we shall always tacitly
assume that $\kappa\log N\in\mathbb N$. In this situation we choose
\[
h=\gamma^{n-1} \log^{n-1} N
\]
for some $\gamma>0$.

\subsection{Willshaw models}

The Willshaw models are close relatives of Amari's models and go back to the
celebrated paper by Willshaw \cite{Willshaw}. They differ from Amari's models in
that the synaptic efficacies do not depend on the number of stored messages using
a given set of neurons, but only on whether this set appears in at least one
stored message.

Throughout this paper, the Willshaw model is considered with the same
Bernoulli pattern ensemble as the Amari model, namely
\[
\mathbb P(\xi_i^\mu=1)=\frac{\log N}{N}.
\]
A closely related alternative, often used for Willshaw-type memories, is to
sample patterns uniformly from configurations with a prescribed number of
active neurons. We do not consider this fixed-weight ensemble here.

The Willshaw-type models have the same dynamics as in
\eqref{multdyn} and \eqref{multdyn_inf}, respectively. The only difference is
that the weights \eqref{multhebb} are replaced by

\begin{equation}\label{willshawhebb}
W_{i_1,\ldots,i_n}:=\mathbf 1\left\{\sum_{\mu=1}^M
\xi_{i_1}^\mu \xi_{i_2}^\mu \cdots \xi_{i_n}^\mu\geq 1\right\}.
\end{equation}

\subsection{GB models}

The next class of models is inspired by the work of Gripon and Berrou (\cite{griponc}, for example). Here the setting is slightly different. We
consider systems consisting of $l$ blocks with $c=c(l)$ neurons per block,
where
\[
c\sim\log l \qquad\text{as } \quad l\to\infty.
\]
Thus the total number of neurons is
$N=lc$.

We now restrict ourselves to a special family of sparse messages
\[
\xi^1,\ldots,\xi^M.
\]

Each message $\xi^\mu$ has exactly one active neuron in each block, and hence
exactly $l$ active neurons in total. More precisely, for each
$\mu=1,\ldots,M$ and each block $a=1,\ldots,l$, let
$K_a^\mu$
be uniformly distributed on $\{1,\ldots,c\}$, independently over all
$a$ and $\mu$. We then set
\[
\xi_{(a,k)}^\mu:=\mathbf 1_{\{K_a^\mu=k\}},
\qquad \text{for} \quad 
a=1,\ldots,l \quad \text{and} \quad k=1,\ldots,c.
\]
Thus every message has exactly one active neuron in each block, and the
choices of the active neurons are independent across blocks and across
messages. 

For an interaction order $n\geq2$, the weights are defined analogously to
those of the Willshaw model. Set

\begin{equation}\label{eq:GBweights}
W_{(a_1,k_1),\ldots,(a_n,k_n)}
:=\mathbf 1\left\{\sum_{\mu=1}^M\xi_{(a_1,k_1)}^\mu\cdots\xi_{(a_n,k_n)}^\mu
\geq 1\right\},
\end{equation}
whenever $a_1,\ldots,a_n$ are pairwise distinct, and set
\[
W_{(a_1,k_1),\ldots,(a_n,k_n)}:=0
\]
otherwise.

For $\sigma\in\{0,1\}^{lc}$, define the local field
\begin{equation}\label{eq:GB_dynamics}
S_{(a,k)}(\sigma):=
\sum_{\substack{1\leq a_2<\cdots<a_n\leq l\\ a_j\neq a,\ 2\leq j\leq n}}
\ \sum_{k_2,\ldots,k_n=1}^{c} W_{(a,k),(a_2,k_2),\ldots,(a_n,k_n)}
\,\sigma_{(a_2,k_2)}\cdots\sigma_{(a_n,k_n)}.
\end{equation}

%
Finally, define the dynamics by
\begin{equation}\label{eq:GB_dynamics2}
T_{(a,k)}(\sigma)
=
\Theta\bigl(S_{(a,k)}(\sigma)-h\bigr),
\end{equation}
for a suitable threshold $h$.

As in the preceding models, we shall consider both fixed and growing
interaction orders. In the latter case we take
\[
n=\kappa c
\]
for a fixed $\kappa>0$, tacitly restricting to values of $l$ for which
$\kappa c \in \mathbb N$.

\section{Storage capacities}

We now turn to the storage capacities of the models introduced in the previous section. Let us first give a definition of 
''storage capacity'': Let $L$ denote the asymptotic size parameter of the model, that is,
$L=N$ for the Amari and Willshaw models and $L=l$ for the Gripon--Berrou model. We call a sequence $(M_L)$ achievable for fixed-pattern stability if
\[
\mathbb P\bigl(T(\xi^1)=\xi^1\bigr)\to 1
\]
as $L\to\infty$, where $T$ denotes the corresponding retrieval map.

For a prescribed reference scale $(a_L)$, we define the
\emph{fixed-pattern storage capacity} 
\[
\alpha_{\mathrm c}^{\mathrm{fp}}:=\sup\left\{
\alpha>0: M_L=\lfloor\alpha a_L\rfloor \text{ is achievable for fixed-pattern stability} \right\}.
\]
Thus $\alpha_{\mathrm c}^{\mathrm{fp}}$ is the largest asymptotic loading parameter, relative to the scale $a_L$, for which a prescribed stored pattern is a fixed point with probability tending to one.

A slightly different and stronger notion is the \emph{simultaneous storage capacity},
\[
\alpha_{\mathrm c}^{\mathrm{all}}
:= \sup\left\{ \alpha>0:
\mathbb P\bigl(T(\xi^\mu)=\xi^\mu\text{ for all }1\leq\mu\leq M_L
\bigr)\to 1
\text{ for }M_L=\lfloor\alpha a_L\rfloor \right\}.
\]
Clearly, $\alpha_{\mathrm c}^{\mathrm{all}} \leq\alpha_{\mathrm c}^{\mathrm{fp}}.$

Throughout the present paper, storage capacity refers to fixed-pattern
storage capacity. For the fixed-order models, our results yield lower bounds
on this capacity relative to their natural storage scales. In the growing-order
regimes, where the loading parameter itself enters with an $N$-dependent
power, we state explicit achievable storage sequences instead. Note, that in the present paper, we do not claim
simultaneous stability of all stored patterns.

Whenever an expression defining the number of stored patterns is not an integer, its integer part is understood.

We first consider fixed interaction order for the Amari, Willshaw, and
Gripon--Berrou models, and then turn to the corresponding growing-order
regimes.

\begin{theorem}\label{theo:amari_fixed_n}
Consider Amari's model with fixed interaction order $n \ge 2$, defined by
\eqref{multdyn} and \eqref{multhebb}, and choose
\[
h=\gamma \log^{n-1} N
\]
for some $0<\gamma<1$. Then there exists $0<\alpha_0=\alpha_0(n,\gamma)>0$ such that
for every $\alpha<\alpha_0$, the choice
\[
M=\alpha \frac{N^n}{(\log N)^n}
\]
satisfies
\[
\mathbb P\bigl(T(\xi^\mu)=\xi^\mu\bigr)\longrightarrow1
\qquad\text{as }N\to\infty
\]
for every fixed stored pattern $\xi^\mu$.
\end{theorem}

Theorem \ref{theo:amari_fixed_n} has an immediate consequence for the Willshaw model with fixed $n$.

\begin{corollary}\label{cor:willshaw_fixed_n}
Consider the higher-order Willshaw model with fixed interaction order $n \ge 2$,
defined by \eqref{multdyn} and \eqref{willshawhebb}, and choose
\[
h=\gamma \log^{n-1} N
\]
for some $0<\gamma<1$. Then there exists $\alpha_0=\alpha_0(n,\gamma)>0$ such that
for every $0<\alpha<\alpha_0$, the choice
\[
M=\alpha \frac{N^n}{(\log N)^n}
\]
satisfies
\[
\mathbb P\bigl(T(\xi^\mu)=\xi^\mu\bigr)\longrightarrow1
\qquad\text{as }N\to\infty
\]
for every fixed stored pattern $\xi^\mu$.
\end{corollary}

We next turn to the Gripon--Berrou architecture at fixed interaction order.
The case $n=2$ corresponds to the classical Gripon--Berrou model and has
already been treated in \cite{GHLV16}. Therefore, we consider a fixed interaction order $n\geq 3$.

\begin{theorem}\label{thm:GB_fixed_n}
Let $n\geq 3$ be fixed and assume that $c=c(l)\sim\log l$. Let
$M=\alpha c^n$ for some $\alpha\in(0,1)$, and choose
\[
h=\gamma\binom{l-1}{n-1}
\]
with $\alpha<\gamma<1$. Then, 
\[
\mathbb P\bigl(T(\xi^\mu)=\xi^\mu\bigr)\longrightarrow1
\qquad\text{as }N\to\infty
\]
for every fixed stored pattern $\xi^\mu$.
\end{theorem}

We now consider the regime $n=\kappa \log N$ of Amari's model.

\begin{theorem}\label{theo:amari_large_n}
Consider Amari's model with interaction order
\[
n=\kappa\log N\in\mathbb N
\]
for some fixed $\kappa\in(0,1)$, and define
\[
\gamma_\kappa :=
\exp\left(\frac{-(1-\kappa)\log(1-\kappa)-\kappa}{\kappa}\right)
=e^{-1}(1-\kappa)^{-(1-\kappa)/\kappa}.
\]
Choose $0<\gamma<\gamma_\kappa$ and let
\[
h=\gamma^{n-1}(\log N)^{n-1}.
\]
Then there exists $\alpha_0=\alpha_0(\kappa,\gamma)>0$ such that, for every
$0<\alpha<\alpha_0$, the choice
\[
M=\alpha^{n-1} \exp\bigl(\kappa(\log^2N-\log N\log\log N)\bigr)
\]
satisfies
\[
\mathbb P\bigl(\widehat T(\xi^\mu)=\xi^\mu\bigr)\longrightarrow1
\qquad\text{as }N\to\infty
\]
for every fixed stored pattern $\xi^\mu$.
\end{theorem}

\begin{rem}
At first sight, the capacity scale
\[
M=\alpha^{n-1}\exp\bigl(\kappa(\log^2N-\log N\log\log N)\bigr)
\]
may seem rather small compared with the exponentially large capacity
\[
M=\exp(\alpha N)
\]
obtained in \cite[Theorem~3]{DHLUV17}. However, the two situations are not directly comparable. In the generalized
Hopfield model the pattern space has cardinality $2^N$, whereas the sparse
patterns considered here have only order $\log N$ active coordinates. For
example,
\[
\log\binom{N}{\lfloor\log N\rfloor}
=
\log^2N-\log N\log\log N+O(\log N).
\]
Thus the number of sparse patterns with about $\log N$ active coordinates has
the exponential scale
\[
\exp\bigl(\log^2N-\log N\log\log N+O(\log N)\bigr).
\]
In this sense, Theorem~\ref{theo:amari_large_n} is the natural sparse analogue
of the large-capacity result in \cite{DHLUV17}.
\end{rem}

The preceding theorem has the following analogue for the Willshaw model. 

\begin{corollary}\label{cor:willshaw_large_n}

Consider the higher-order Willshaw model with interaction order
\[
n=\kappa\log N\in\mathbb N,
\qquad \kappa\in(0,1),
\]
defined by \eqref{multdyn_inf} and \eqref{willshawhebb}. Let
$\gamma_\kappa$ be as in Theorem~\ref{theo:amari_large_n}, choose
$0<\gamma<\gamma_\kappa$,
and set
\[
h=\gamma^{n-1}(\log N)^{n-1}.
\]
Then there exists $\alpha_0=\alpha_0(\kappa,\gamma)>0$ such that for every
$0<\alpha<\alpha_0$, the choice
\[
M=\alpha^{n-1}
\exp\bigl(\kappa(\log^2N-\log N\log\log N)\bigr)
\]
satisfies
\[
\mathbb P\bigl(\widehat T(\xi^\mu)=\xi^\mu\bigr)\longrightarrow1
\qquad\text{as }N\to\infty
\]
for every fixed stored pattern $\xi^\mu$.
\end{corollary}

We finally consider the Gripon--Berrou model with growing interaction order
$n=\kappa c$. In this regime we obtain the following result.
\begin{theorem}\label{thm:GB_logarithmic_n}
Assume that $c=c(l)\sim \log l$ and let
$n=\kappa c$
for some fixed $\kappa>0$. Let
\[
M=\alpha c^n
\]
for some $\alpha\in(0,1)$, and choose
\[
h=\gamma \binom{l-1}{n-1}
\]
with $\alpha<\gamma<1$.
Then, for every fixed message $\xi^\mu$,
\[
\mathbb P\bigl(T(\xi^\mu)=\xi^\mu\bigr)\longrightarrow 1
\qquad\text{as } l\to\infty.
\]
\end{theorem}


\section{Proofs}
We start with the proof of Theorem \ref{theo:amari_fixed_n}.

\begin{proof}[Proof of Theorem \ref{theo:amari_fixed_n}]
Without loss of generality, it suffices to consider $\mu=1$.

Choose $\delta\in(0,1)$ such that
$\delta^{n-1}>\gamma$,
and define
\begin{equation}\label{eq:Adelta_new}
	A_\delta:=\left\{\left|\sum_{j=1}^N \xi_j^1-\log N\right|
	\leq (1-\delta)\log N\right\}.
\end{equation}	
Since
\[
\mathbb E\left[\sum_{j=1}^N\xi_j^1\right]=\log N
\qquad\text{and}\qquad \operatorname{Var}\left(\sum_{j=1}^N\xi_j^1\right)\leq\log N,
\]
Chebyshev's inequality yields
\[
\mathbb P(A_\delta^c)
\leq \frac{1}{(1-\delta)^2\log N}
\to 0.
\]

We first consider coordinates $i$ such that $\xi_i^1=1$. On the event $A_\delta$,
the pattern $\xi^1$ has at least $\lceil\delta\log N\rceil$ active coordinates.
Since $i$ itself is active, there are at least
$\lceil\delta\log N\rceil-1$ further active coordinates. Hence the contribution
of the message $\mu=1$ alone yields
\[
\sum_{1\le j_1,\ldots,j_{n-1}\le N}^{*}
\xi_{j_1}^1\cdots \xi_{j_{n-1}}^1
W_{i,j_1,\ldots,j_{n-1}}
\ge
(n-1)!\binom{\lceil\delta\log N\rceil-1}{n-1}.
\]
For every $\varepsilon>0$ and all sufficiently large $N$,
\begin{equation}\label{eq:amari_pos_lower}
	(n-1)!\binom{\lceil\delta\log N\rceil-1}{n-1} \ge
	(1-\varepsilon)\delta^{n-1}(\log N)^{n-1}.
\end{equation}
Since $\delta^{n-1}>\gamma$, we may choose $\varepsilon>0$ sufficiently small
such that
\[
(1-\varepsilon)\delta^{n-1}>\gamma.
\]
It follows that, for all sufficiently large $N$,
\[
T_i(\xi^1)=1 \qquad\text{on }A_\delta
\]
whenever $\xi_i^1=1$.

Next consider a coordinate $i$ such that $\xi_i^1=0$. Define the support of
$\xi^1$ outside the coordinate $i$ by
\[
\mathcal S_i(\xi^1):= \bigl\{j\in\{1,\ldots,N\}\setminus\{i\}:\xi_j^1=1\bigr\}.
\]
On the event $\{\xi_i^1=0\}$, we have
$|\mathcal S_i(\xi^1)|=\sum_{j=1}^N\xi_j^1$.
Let
\[
d_-:=\left\lceil\delta\log N\right\rceil,
\qquad \text{and} \quad
d_+:=\left\lfloor(2-\delta)\log N\right\rfloor.
\]
Thus, on $\{\xi_i^1=0\}\cap A_\delta$, the cardinality of
$\mathcal S_i(\xi^1)$ lies between $d_-$ and $d_+$.

For a set $A\subset\{1,\ldots,N\}\setminus\{i\}$, define
\[
H_i(A):=
\sum_{\substack{(i_1,\ldots,i_{n-1})\in A^{n-1}\\
		i_1,\ldots,i_{n-1}\ \mathrm{pairwise\ distinct}}}
\sum_{\mu=2}^M \xi_i^\mu\xi_{i_1}^\mu\cdots\xi_{i_{n-1}}^\mu.
\]
If $\xi_i^1=0$, the contribution of the stored pattern $\mu=1$ to the local field at $i$ vanishes. Consequently, for every deterministic set
$A\subset\{1,\ldots,N\}\setminus\{i\}$,
\[
\mathbb P\left(T_i(\xi^1)=1\,\middle|\,\xi_i^1=0,\ \mathcal S_i(\xi^1)=A\right)
=\mathbb P\bigl(H_i(A)>h\bigr),
\]
where the probability on the right-hand side is taken only with respect to the
patterns $\xi^2,\ldots,\xi^M$.

Since the joint distribution of the patterns $\xi^\mu$, $\mu\geq2$, is
invariant under permutations of the coordinates that leave $i$ fixed, the
probability $\mathbb P(H_i(A)>h)$ depends on $A$ only through its cardinality. For $d'\in\{0,\ldots,N-1\}$, let
\[
q_{d'} := \mathbb P\bigl(H_i(A)>h\bigr),
\]
where $A$ is any deterministic subset of
$\{1,\ldots,N\}\setminus\{i\}$ with $|A|=d'$.

It follows from the law of total probability that
\begin{align}
	&\mathbb P\bigl( T_i(\xi^1)\neq\xi_i^1,\ \xi_i^1=0,\ A_\delta
	\bigr)
	\nonumber\\
	&=\sum_{d'=d_-}^{d_+} \sum_{\substack{A\subset\{1,\ldots,N\}\setminus\{i\}\\|A|=d'}}
	\mathbb P\bigl( \xi_i^1=0,\ \mathcal S_i(\xi^1)=A\bigr)
	\mathbb P\bigl(H_i(A)>h\bigr)
	\nonumber\\
	&=\sum_{d'=d_-}^{d_+} q_{d'}\,
	\mathbb P\bigl(\xi_i^1=0,\ |\mathcal S_i(\xi^1)|=d'\bigr).
	\label{eq:amari_zero_conditioned}
\end{align}
The quantities $q_{d'}$ are nondecreasing in $d'$. Indeed, we may choose deterministic sets
\[
A_{d_-}\subset A_{d_-+1}\subset\cdots\subset A_{d_+}, \qquad \text{with  } |A_{d'}|=d',
\]
and then
\[
H_i(A_{d'})\leq H_i(A_{d'+1})
\]
because all summands are nonnegative. Therefore,
\[
\mathbb P\bigl( T_i(\xi^1)\neq\xi_i^1,\ \xi_i^1=0,\ A_\delta
\bigr) \leq q_{d_+}.
\]

Set $d:=d_+$. By permutation invariance, we may relabel the coordinates so that
$i=N$ and choose $A_d=\{1,\ldots,d\}$.
Hence
\begin{align}
	&\mathbb P\bigl(T_i(\xi^1)\neq\xi_i^1,\ \xi_i^1=0,\ A_\delta\bigr)
	\leq\mathbb P\left(
	\sum_{1\leq i_1,\ldots,i_{n-1}\leq d}^{*}\sum_{\mu=2}^M
	\xi_i^\mu\xi_{i_1}^\mu\cdots\xi_{i_{n-1}}^\mu>h\right).
	\label{eq:amari_zero_reduced_final}
\end{align}

Using Markov's inequality, we obtain for every $t>0$,
\begin{multline}
	\mathbb P\left(
	\sum_{i_1,\ldots,i_{n-1}\le d}^{*}\sum_{\mu=2}^M
	\xi_i^\mu\xi_{i_1}^\mu\cdots\xi_{i_{n-1}}^\mu
	\ge h
	\right)
	\\\le e^{-th}
	\mathbb E\left[
	\exp\left(
	t\sum_{i_1,\ldots,i_{n-1}\le d}^{*}\sum_{\mu=2}^M
	\xi_i^\mu\xi_{i_1}^\mu\cdots\xi_{i_{n-1}}^\mu
	\right)
	\right].
	\label{eq:amari_markov_final}
\end{multline}
By independence of the patterns,
\begin{align}
	&
	\mathbb E\left[
	\exp\left(
	t\sum_{i_1,\ldots,i_{n-1}\le d}^{*}\sum_{\mu=2}^M
	\xi_i^\mu\xi_{i_1}^\mu\cdots\xi_{i_{n-1}}^\mu
	\right)
	\right]
	=
	\mathbb E\left[
	\exp\left(
	t\sum_{i_1,\ldots,i_{n-1}\le d}^{*}
	\xi_i^2\xi_{i_1}^2\cdots\xi_{i_{n-1}}^2
	\right)
	\right]^{M-1}
	\nonumber\\
	\le&
	\mathbb E\left[
	\exp\left(
	t\sum_{i_1,\ldots,i_{n-1}\le d}^{*}
	\xi_i^2\xi_{i_1}^2\cdots\xi_{i_{n-1}}^2
	\right)
	\right]^M,
	\label{eq:amari_pattern_factor}
\end{align}
since the expectation on the right-hand side is at least $1$. Taking expectation
with respect to $\xi_i^2$ first, we get
\begin{multline}
	\mathbb E\left[
	\exp\left(
	t\sum_{i_1,\ldots,i_{n-1}\le d}^{*}
	\xi_i^2\xi_{i_1}^2\cdots\xi_{i_{n-1}}^2
	\right)
	\right]
	\\=	1-p+
	p\,\mathbb E\left[
	\exp\left(
	t\sum_{i_1,\ldots,i_{n-1}\le d}^{*}
	\xi_{i_1}^2\cdots\xi_{i_{n-1}}^2
	\right)
	\right].
	\label{eq:amari_xi_i_first}
\end{multline}

Set $x_j:=\xi_j^2$ for $j=1,\dots,d$, and define
\[
h(x_{i_1},\ldots,x_{i_{n-1}}):=x_{i_1}\cdots x_{i_{n-1}}.
\]
Then
\[
\sum_{i_1,\ldots,i_{n-1}\le d}^{*}\xi_{i_1}^2\cdots\xi_{i_{n-1}}^2
=
(n-1)!\binom{d}{n-1}U_d(h),
\]
where
\[
U_d(h):=
\frac{1}{\binom{d}{n-1}}
\sum_{1\le i_1<\cdots<i_{n-1}\le d}
h(x_{i_1},\ldots,x_{i_{n-1}})
\]
is a $U$-statistic of degree $n-1$.

We now use Hoeffding's decomposition idea \cite{Hoeffding1963}. Set
\[
k:=\left\lfloor\frac{d}{n-1}\right\rfloor
\]
and
\[
B(x_1,\ldots,x_d):=
\frac1k\sum_{r=0}^{k-1}
h(x_{r(n-1)+1},\ldots,x_{(r+1)(n-1)}).
\]
Since the kernel $h$ is symmetric, averaging each of the $k$ disjoint blocks
in $B$ over all permutations of the coordinates yields the same average over
all $(n-1)$-subsets and hence,
\[
U_d(h)=\frac1{d!}\sum_{\sigma\in S_d}
B(x_{\sigma(1)},\ldots,x_{\sigma(d)}).
\]
where, of course, $S_d$ is the symmetric group over $\{1,\ldots, d\}$. Hence, for any $T\in \R$
\begin{multline*}
	\exp(T U_d(h)) = \exp\left(T\frac 1 {d!} \sum_{\sigma \in S_d} B(x_{\sigma(1)}, \ldots, x_{\sigma(d)})\right) \\ \le \frac 1 {d!} \sum_{\sigma \in S_d} \exp\left(T B(x_{\sigma(1)}, \ldots, x_{\sigma(d)})\right)
\end{multline*}
by Jensen's inequality. 
Since the blocks in the definition of $B$ are disjoint, the corresponding random
variables are independent, and therefore
\begin{align}
	\mathbb E[e^{TU_d(h)}]
	&\le
	\prod_{r=0}^{k-1}
	\mathbb E\left[
	\exp\left(
	\frac{T}{k}h(x_{r(n-1)+1},\ldots,x_{(r+1)(n-1)})
	\right)
	\right]
	\nonumber\\
	&=
	\left(
	\mathbb E\left[
	\exp\left(\frac{T}{k}h(x_1,\ldots,x_{n-1})\right)
	\right]
	\right)^k.
	\label{eq:amari_ustat_block}
\end{align}
Now $h(x_1,\ldots,x_{n-1})$ is Bernoulli distributed with success probability
$p^{n-1}$. Hence
\[
\mathbb E\left[
\exp\left(\frac{T}{k}h(x_1,\ldots,x_{n-1})\right)
\right]
=
1-p^{n-1}+p^{n-1}e^{T/k}.
\]
Therefore
\[
\mathbb E[e^{TU_d(h)}]
\le
\left(1-p^{n-1}+p^{n-1}e^{T/k}\right)^k.
\]

We now choose
\[
T=t (n-1)!\binom{d}{n-1}.
\]
Since $n$ is fixed, there exists a constant $C_n>0$ such that
\[
\frac{(n-1)!\binom{d}{n-1}}{k}\le C_n d^{n-2}
\]
for all large $d$. Thus
\begin{align}
	&
	\mathbb E\left[
	\exp\left(
	t\sum_{i_1,\ldots,i_{n-1}\le d}^{*}
	\xi_{i_1}^2\cdots\xi_{i_{n-1}}^2
	\right)
	\right]
	\le		\left(1-p^{n-1}+p^{n-1}e^{C_n t d^{n-2}}\right)^k.
	\label{eq:amari_ustat_mgf}
\end{align}
Combining \eqref{eq:amari_markov_final}, \eqref{eq:amari_pattern_factor},
\eqref{eq:amari_xi_i_first}, and \eqref{eq:amari_ustat_mgf}, we obtain
\begin{multline}
	\mathbb P\left(
	\sum_{i_1,\ldots,i_{n-1}\le d}^{*}\sum_{\mu=2}^M
	\xi_i^\mu\xi_{i_1}^\mu\cdots\xi_{i_{n-1}}^\mu
	\ge h
	\right)
	\\ \le	\exp\left(
	-th
	+
	Mp\left[
	\left(1-p^{n-1}+p^{n-1}e^{C_n t d^{n-2}}\right)^k-1
	\right]
	\right).
	\label{eq:amari_main_exp_final}
\end{multline}

Now choose $t=\frac{a}{d^{n-2}}$
with $a>0$ to be fixed later. Then
$C_n t d^{n-2}=C_n a$,
and therefore
\[
\left(1-p^{n-1}+p^{n-1}e^{C_n t d^{n-2}}\right)^k
=
\left(1+p^{n-1}(e^{C_n a}-1)\right)^k.
\]
Since
\[
k\,p^{n-1}\asymp d\left(\frac{\log N}{N}\right)^{n-1}\longrightarrow0,
\]
we have
\[
\left(1+p^{n-1}(e^{C_n a}-1)\right)^k-1
=
(1+o(1))\,k\,p^{n-1}(e^{C_n a}-1).
\]
Hence the exponent in \eqref{eq:amari_main_exp_final} is bounded by
\begin{align}
	-th + (1+o(1))Mp^n k (e^{C_n a}-1).
	\label{eq:amari_exponent_prelim}
\end{align}
Now recall that 
\[
h=\gamma (\log N)^{n-1}\qquad \text{and} \quad d=(2-\delta)\log N +O(1),
\]
and therefore
\[
th=\frac{a}{d^{n-2}}\gamma(\log N)^{n-1}
=(1+o(1))\frac{a\gamma}{(2-\delta)^{n-2}}\log N.
\]
Moreover,
\[
Mp^n=\left(\alpha\frac{N^n}{(\log N)^n}+O(1)\right)
\frac{(\log N)^n}{N^n}=\alpha+o(1).
\]
and
\[
k=(1+o(1))\frac{d}{n-1}
=(1+o(1))\frac{2-\delta}{n-1}\log N.
\]
Thus \eqref{eq:amari_exponent_prelim} becomes
\[
-(1+o(1))\frac{a\gamma}{(2-\delta)^{n-2}}\log N
+(1+o(1))\alpha\frac{2-\delta}{n-1}(e^{C_n a}-1)\log N.
\]

Set
\[
A(a):=\frac{a\gamma}{(2-\delta)^{n-2}}
\qquad\text{and}\qquad B(a,\alpha):=
\alpha\frac{2-\delta}{n-1}\bigl(e^{C_na}-1\bigr).
\]
Choose first $a>0$ sufficiently large so that $
A(a)>1$.
Having fixed such an $a$, choose $\alpha_0>0$ so small that
$B(a,\alpha_0)
<\frac{A(a)-1}{2}$.
Then, for every $0<\alpha<\alpha_0$,
$B(a,\alpha)<\frac{A(a)-1}{2}$, as well, 
and hence
\[
A(a)-B(a,\alpha)>\frac{A(a)+1}{2}>1.
\]
Consequently, there exists $\rho>0$ such that
\[
A(a)-B(a,\alpha)>1+2\rho.
\]
Since the exponent in \eqref{eq:amari_main_exp_final} is
\[
-\bigl(A(a)-B(a,\alpha)+o(1)\bigr)\log N,
\]
it is bounded above by $-(1+\rho)\log N$
for all sufficiently large $N$. Therefore,
\[
\mathbb P\bigl(T_i(\xi^1)\neq\xi_i^1,\ \xi_i^1=0,\ A_\delta
\bigr)\leq N^{-(1+\rho)}.
\]
Combining this with the deterministic stability of the active coordinates on $A_\delta$, we conclude by a union bound that
\[
\mathbb P\bigl(T(\xi^1)\neq\xi^1,\ A_\delta\bigr)
\leq N\cdot N^{-(1+\rho)}=N^{-\rho}
\to 0.
\]
Since $\mathbb P(A_\delta^c)\to0$, it follows that
\[
\mathbb P\bigl(T(\xi^1)=\xi^1\bigr)\longrightarrow 1.
\]

\end{proof}

Let us next turn to the proof of Corollary \ref{cor:willshaw_fixed_n}:

\begin{proof}[Proof of Corollary \ref{cor:willshaw_fixed_n}]
Choose $\delta\in(0,1)$ such that
$\delta^{n-1}>\gamma$,
and consider again the event $A_\delta$ defined in
\eqref{eq:Adelta_new}.

First consider a coordinate $i$ such that $\xi_i^1=1$. On the event $A_\delta$,
the message $\xi^1$ itself creates all relevant $(n-1)$-tuples of active
coordinates, and therefore the local field in the Willshaw model is bounded below by
\[
\sum_{1\le j_1,\ldots,j_{n-1}\le N}^{*}
\xi_{j_1}^1\cdots \xi_{j_{n-1}}^1,
\]
exactly as in the Amari model. Hence, for the same choice of $\delta$, every
active coordinate of $\xi^1$ remains stable for all sufficiently large $N$.

Now let $i$ be such that $\xi_i^1=0$. For every choice of pairwise distinct indices $i_1,\ldots,i_{n-1}$ different from $i$, the contribution of the Willshaw weight is
\[
\mathbf 1\left\{\sum_{\mu=2}^M \xi_i^\mu\xi_{i_1}^\mu\cdots\xi_{i_{n-1}}^\mu
\geq 1 \right\}.
\]
Since
\[
\mathbf 1\left\{\sum_{\mu=2}^M \xi_i^\mu\xi_{i_1}^\mu\cdots\xi_{i_{n-1}}^\mu
\geq 1 \right\}
\leq\sum_{\mu=2}^M\xi_i^\mu\xi_{i_1}^\mu\cdots\xi_{i_{n-1}}^\mu,
\]
the local field at an inactive coordinate in the Willshaw model is bounded above by the corresponding local field in the Amari model. Hence, if an inactive coordinate is turned into a $1$ by the Willshaw dynamics, it is also turned into a $1$ by the Amari dynamics.

Therefore the error probability for inactive coordinates in the Willshaw model is
bounded by the corresponding error probability in Theorem~\ref{theo:amari_fixed_n}.
Combining this with the stability of active coordinates on $A_\delta$ and using
$\mathbb P(A_\delta)\to1$, the claim follows.
\end{proof}

We now turn to the proof of Theorem \ref{thm:GB_fixed_n}.
\begin{proof}[Proof of Theorem \ref{thm:GB_fixed_n}]

By symmetry, it suffices to consider the stored message $\xi^1$. We condition
on $\xi^1$ throughout the proof. Since every message contains exactly one active
neuron in each block, the estimates below do not depend on the particular
realization of $\xi^1$.

Set $r:=n-1$ and
$K:=\binom{l-1}{r}$.
For every block $a$, let $(a,k_a^1)$ denote the unique active neuron of $\xi^1$
in that block. Since every hyperedge containing $(a,k_a^1)$ and the active
neurons of $\xi^1$ in any $r$ other blocks is generated by the message $\xi^1$
itself, we have
\[
S_{(a,k_a^1)}(\xi^1)=K> \gamma K=h.
\]
Thus all active neurons of $\xi^1$ are stable deterministically.

It remains to consider a false neuron $(a,k)$, that is to say, $k\neq k_a^1$. For
$\nu\geq2$, define $B_\nu:=\xi^\nu_{(a,k)}$
and let
\[
X_\nu:=\sum_{b\neq a}\xi^\nu_{(b,k_b^1)}.
\]
Then  $B_\nu$ indicates whether the message $\xi^\nu$ activates the fixed neuron
$(a,k)$ and thus
$B_\nu$ is Bernoulli distributed with parameter $1/c$, while
\begin{equation} \label{eq:def_Xnu}
	X_\nu\sim\operatorname{Bin}\left(l-1,\frac1c\right)
	\qquad \text{with}\qquad 
	\lambda:=\mathbb E[X_\nu]=\frac{l-1}{c}.
\end{equation}
Moreover, $B_\nu$ and $X_\nu$ are independent.

For a fixed set $A\subset\{1,\ldots,l\}\setminus\{a\}$ with $|A|=r$, the
corresponding clipped weight is one if at least one stored message generates
the hyperedge. Since $\xi^1_{(a,k)}=0$, the message $\xi^1$ itself cannot
generate such a hyperedge. Hence
\[
W_{(a,k),((b,k_b^1))_{b\in A}}
\leq
\sum_{\nu=2}^M B_\nu\prod_{b\in A}\xi^\nu_{(b,k_b^1)}.
\]
Summing over all $r$-subsets $A\subset\{1,\ldots,l\}\setminus\{a\}$ and
observing that, for fixed $\nu$, exactly $\binom{X_\nu}{r}$
such subsets are contained in the $X_\nu$ matching blocks, we obtain
\[
S_{(a,k)}(\xi^1)
\leq
\sum_{\nu=2}^M B_\nu\binom{X_\nu}{r}.
\]
Choose $\varepsilon\in(0,1)$ and $\eta>0$ sufficiently small such that
$(1+2\varepsilon)^r(1+\eta)\alpha<\gamma$,
which is possible since $\alpha<\gamma$. Define
\[
E_\nu:=\{|X_\nu-\lambda|\leq\varepsilon\lambda\}
\qquad \text{and} \qquad
E:=\bigcap_{\nu=2}^M E_\nu.
\]
Due to \eqref{eq:def_Xnu}
%
the multiplicative Chernoff bounds yield
\[
\mathbb P\bigl(X_\nu\geq(1+\varepsilon)\lambda\bigr)
\leq \exp\left(-\frac{\varepsilon^2}{3}\lambda\right)
\qquad \text{and}\quad
\mathbb P\bigl(X_\nu\leq(1-\varepsilon)\lambda\bigr)
\leq \exp\left(-\frac{\varepsilon^2}{2}\lambda\right).
\]
Hence
\[
\mathbb P(E_\nu^c)
\leq
2\exp\left(-\frac{\varepsilon^2}{3}\lambda\right).
\]
Therefore,
\[
\mathbb P(E^c)\leq 2M\exp\left(-\frac{\varepsilon^2}{3}\lambda\right).
\]
On $E_\nu$, we have
$X_\nu\leq(1+\varepsilon)\lambda=(1+\varepsilon)\frac{l-1}{c}$.
Hence
\[
\binom{X_\nu}{r}\leq\frac{X_\nu^r}{r!}\leq(1+\varepsilon)^r\frac{(l-1)^r}{r!c^r}.
\]
Since $r$ is fixed and
\[
\frac{(l-1)^r/r!}{K} =\frac{(l-1)^r}{(l-1)_r} \to 1,
\]
we obtain, for all sufficiently large $l$,
\[
\binom{X_\nu}{r} \leq(1+2\varepsilon)^rKc^{-r}.
\]
Therefore, on the set $E$, we have
\[
S_{(a,k)}(\xi^1)
\leq(1+2\varepsilon)^rKc^{-r}\sum_{\nu=2}^M B_\nu.
\]
Now set $Y:=\sum_{\nu=2}^M B_\nu$.
Then
\[
Y\sim\operatorname{Bin}\left(M-1,\frac1c\right)
\qquad \text{and}\qquad
\mu_Y:=\mathbb E[Y]=\frac{M-1}{c}\sim\frac{M}{c}.
\]
Since
\[
(1+\eta)\frac{M}{c}\geq(1+\eta)\mu_Y,
\]
the multiplicative Chernoff inequality yields
\[
\mathbb P\left(Y>(1+\eta)\frac{M}{c}\right) \leq \mathbb P\bigl(Y>(1+\eta)\mu_Y\bigr)
\leq \exp\left(-\frac{\eta^2}{2+\eta}\mu_Y\right).
\]
For all sufficiently large $l$, we have $M-1\geq M/2$, and hence we arrive at
\[
\mathbb P\left(Y>(1+\eta)\frac{M}{c}\right)
\leq \exp\left(-C_\eta\frac{M}{c}\right)
\]
with $C_\eta:=\frac{\eta^2}{2(2+\eta)}>0.$

On the event
\[
E\cap\left\{Y\leq(1+\eta)\frac{M}{c}\right\},
\]
we use $M\leq\alpha c^n=\alpha c^{r+1}$ and obtain
\[
S_{(a,k)}(\xi^1)\leq (1+2\varepsilon)^rKc^{-r}(1+\eta)\frac{M}{c}
\leq(1+2\varepsilon)^r(1+\eta)\alpha K<\gamma K=h.
\]
Hence, conditionally on $\xi^1$, for every false neuron $(a,k)$,
\[
\mathbb P\bigl(T_{(a,k)}(\xi^1)=1\,\big|\,\xi^1\bigr)
\leq
2M\exp\left(-\frac{\varepsilon^2}{3}\lambda\right)
+\exp\left(-C_\eta\frac{M}{c}\right).
\]
Since $c\sim\log l$, we have
\[
\lambda=\frac{l-1}{c}\sim\frac{l}{\log l} \qquad\text{and}\qquad M=O(c^n),
\]
and therefore $2M\exp\left(-\frac{\varepsilon^2}{3}\lambda\right)=o((lc)^{-1})$.
Furthermore,
\[
\frac{M}{c}=(\alpha+o(1))c^{n-1}.
\]
Since $n\geq 3$ and $c^{n-1}\asymp(\log l)^{n-1}$, we have
\[
\exp\left(-C_\eta\frac{M}{c}\right)=\exp\bigl(-(C_\eta\alpha+o(1))c^{n-1}\bigr) = o((lc)^{-1}),
\]
Therefore, uniformly in the realization of $\xi^1$ and in the choice of the false neuron,
\[
\mathbb P\bigl(T_{(a,k)}(\xi^1)=1\,\big|\,\xi^1\bigr) = o((lc)^{-1}).
\]
Every realization of $\xi^1$ contains exactly $l(c-1)$ false neurons, while
all its active neurons are stable deterministically. Hence, by a union bound,
\[
\mathbb P\bigl(T(\xi^1)\neq\xi^1\,\big|\,\xi^1\bigr)
\leq l(c-1)\,o((lc)^{-1}) =o(1),
\]
uniformly in $\xi^1$. Averaging over $\xi^1$ therefore yields
\[
\mathbb P\bigl(T(\xi^1)\neq\xi^1\bigr)\longrightarrow0.
\]
This proves the theorem.
\end{proof}

Let us now turn to the situation where in Amari's model we choose
$n=\kappa\log N\in\mathbb N$ for some $\kappa\in(0,1)$.
\begin{proof}[Proof of Theorem \ref{theo:amari_large_n}]
The proof is simpler than in the fixed-order $n$ case. Indeed, for fixed $n$, the
false field is a sum of many small but non-negligible contributions, and its
control requires an exponential-moment estimate. When $n=\kappa\log N$, however,
the relevant coincidence probabilities become so small that a first-moment bound
already suffices.

As before, it suffices to consider a fixed stored pattern, say $\xi^1$.

Define
\begin{equation}\label{eq:gamma_kappa_delta}
	\Gamma_\kappa(\delta):=
	\exp\left(\frac{1}{\kappa}\int_0^\kappa\log(\delta-s)\,ds
	\right)
	=\exp\left(\frac{\delta\log\delta-(\delta-\kappa)\log(\delta-\kappa)
		-\kappa}{\kappa}\right),
\end{equation}
for $\delta>\kappa$.
The function $\delta\mapsto\Gamma_\kappa(\delta)$ is continuous and strictly increasing on $(\kappa,\infty)$. Since
\[
\gamma<\gamma_\kappa:=\Gamma_\kappa(1)
=e^{-1}(1-\kappa)^{-(1-\kappa)/\kappa},
\]
we may therefore choose $\delta\in(\kappa,1)$ such that
$\gamma<\Gamma_\kappa(\delta)$.

Let
\[
A_\delta:=\left\{\left|\sum_{j=1}^N\xi_j^1-\log N\right|
\leq(1-\delta)\log N\right\}.
\]
Then
\[
\mathbb P(A_\delta)\longrightarrow1
\qquad\text{as }N\to\infty.
\]
We first consider a coordinate $i$ such that $\xi_i^1=1$. On the event
$A_\delta$, the pattern $\xi^1$ has at least $\delta\log N$ active
coordinates. Set
\[
L:=\log N,\qquad m:=n-1,\qquad \text{and} \quad r_N:=\lceil\delta L\rceil-1.
\]
Since the total number of active coordinates is an integer, it is at least
$\lceil\delta L\rceil$ on $A_\delta$. After excluding the coordinate $i$,
there therefore remain at least $r_N$ active coordinates. Hence the
contribution of the pattern $\mu=1$ alone is bounded from below by
the falling factorial $(r_N)_m.$

Moreover, since $\delta>\kappa$ and $m=\kappa L-1$, we have $r_N\geq m$ for
all sufficiently large $N$. Thus all factors in $(r_N)_m$ are positive, and
\begin{align}
	\frac{1}{m}\log\left(\frac{(r_N)_m}{L^m}\right)
	&=\frac{1}{m}\sum_{q=0}^{m-1}\log\left(\frac{r_N-q}{L}\right)
	\longrightarrow \frac{1}{\kappa}\int_0^\kappa\log(\delta-s)\,ds
	=\log\Gamma_\kappa(\delta).
	\label{eq:falling_factorial_asymptotics}
\end{align}
Choose $\widetilde\gamma$ such that
\[
\gamma<\widetilde\gamma<\Gamma_\kappa(\delta).
\]
It follows from \eqref{eq:falling_factorial_asymptotics} that, for all
sufficiently large $N$,
$(r_N)_m \geq\widetilde\gamma^{\,m}L^m>\gamma^mL^m=h$.
Consequently,
\[
\widehat T_i(\xi^1)=1
\qquad\text{on }A_\delta
\]
for all sufficiently large $N$.

Now consider a coordinate $i$ such that $\xi_i^1=0$. As in the proof of Theorem~\ref{theo:amari_fixed_n}, define
\[
\mathcal S_i(\xi^1)
:= \bigl\{j\in\{1,\ldots,N\}\setminus\{i\}:\xi_j^1=1\bigr\}.
\]
On the event $\{\xi_i^1=0\}\cap A_\delta$, we have
\[
|\mathcal S_i(\xi^1)|\leq d,
\qquad \text{with}\qquad 
d:=\left\lfloor(2-\delta)\log N\right\rfloor.
\]
For a deterministic set
$A\subset\{1,\ldots,N\}\setminus\{i\}$, define
\[
H_i(A)
:=\sum_{\substack{(i_1,\ldots,i_{n-1})\in A^{n-1}\\
		i_1,\ldots,i_{n-1}\ \mathrm{pairwise\ distinct}}}
\sum_{\mu=2}^M\xi_i^\mu\xi_{i_1}^\mu\cdots\xi_{i_{n-1}}^\mu.
\]
If $\xi_i^1=0$, the contribution of the stored pattern $\mu=1$ to the
local field at $i$ vanishes, since every corresponding summand contains the factor $\xi_i^1$. Therefore, conditional on
$\xi_i^1=0$ and $\mathcal S_i(\xi^1)=A$,
the false local field at $i$ is equal to $H_i(A)$. The randomness in
$H_i(A)$ comes only from the independent patterns
$\xi^2,\ldots,\xi^M$.

Since all summands in $H_i(A)$ are nonnegative, we have
\[
H_i(A)\leq H_i(\widetilde A)
\qquad\text{whenever}\qquad A\subseteq\widetilde A.
\]
Therefore, for every $A$ with $|A|\leq d$, we may enlarge $A$ to a
deterministic set $A_d$ of cardinality $d$ and obtain
\[
\mathbb P\bigl(H_i(A)>h\bigr)
\leq \mathbb P\bigl(H_i(A_d)>h\bigr).
\]
By permutation invariance of the patterns $\xi^\mu$, $\mu\geq2$, under
permutations that leave $i$ fixed, the probability on the right-hand side does not depend on the particular choice of $A_d$. We may therefore relabel the coordinates and take
$i=N$ and $A_d=\{1,\ldots,d\}$.
Conditioning on $\mathcal S_i(\xi^1)$ and using the preceding monotonicity
bound for every possible realization of the support, we obtain
\[
\begin{split}
	&\mathbb P\bigl(\widehat T_i(\xi^1)\neq\xi_i^1,\ \xi_i^1=0,\ A_\delta\bigr)\\
	&\quad=\sum_{\substack{A\subset\{1,\ldots,N\}\setminus\{i\}\\ |A|\leq d}}\mathbb P\bigl(\xi_i^1=0,\ \mathcal S_i(\xi^1)=A,\ A_\delta\bigr)
	\mathbb P\bigl(H_i(A)>h\bigr)\\
	&\quad\leq \mathbb P\bigl(H_i(A_d)>h\bigr).
\end{split}
\]
For $i=N$ and $A_d=\{1,\ldots,d\}$, this gives
\[
\mathbb P\bigl(\widehat T_i(\xi^1)\neq\xi_i^1,\ \xi_i^1=0,\ A_\delta\bigr)
\leq\mathbb P\left(\sum_{\mu=2}^M\xi_i^\mu(X_\mu)_{n-1}>h\right),
\]
where
\[
X_\mu:=\sum_{j=1}^d\xi_j^\mu
\sim\operatorname{Bin}(d,p),
\qquad \text{with} \quad p=\frac{\log N}{N}.
\]
Indeed, $X_\mu$ counts the number of coordinates in $A_d$ at which
$\xi^\mu$ is active, and hence $(X_\mu)_{n-1}$
is the number of ordered $(n-1)$-tuples of pairwise distinct active
coordinates contributed by the pattern $\xi^\mu$. As usual,
$(X_\mu)_{n-1}=0$ if $X_\mu<n-1$.

Since $i=N\notin A_d$, the random variables $\xi_i^2$ and $X_2$ are
independent. Moreover, the factorial-moment identity for a binomial random
variable gives
\[
\mathbb E[(X_2)_{n-1}]=(d)_{n-1}p^{n-1}.
\]

Applying Markov's inequality and writing $\mathbb E_{2:M}$ for expectation
with respect to the patterns $\xi^2,\ldots,\xi^M$, we obtain
\begin{align}
	&\mathbb P\bigl(\widehat T_i(\xi^1)\neq\xi_i^1,\ \xi_i^1=0,\ A_\delta
	\bigr)\leq \frac{1}{h}\,\mathbb E_{2:M}\left[\sum_{\mu=2}^M\xi_i^\mu(X_\mu)_{n-1}\right]
	\nonumber\\
	&=\frac{M-1}{h}\,\mathbb E[\xi_i^2]\,\mathbb E[(X_2)_{n-1}]=\frac{M-1}{h}\,p\,(d)_{n-1}p^{n-1}
	\nonumber\\
	&=\frac{M-1}{h}\,(d)_{n-1}p^n.
	\label{eq:amari_large_first_moment}
\end{align}

Now, by the definition of $M$,
\[
Mp^n\leq\alpha^{n-1} \exp\bigl(\kappa(\log^2N-\log N\log\log N)\bigr)
\left(\frac{\log N}{N}\right)^n
=\alpha^{n-1},
\]
and
\[
(d)_{n-1}\leq d^{n-1}\leq ((2-\delta)\log N)^{n-1}.
\]
Since
\[
h=\gamma^{n-1}\log^{n-1}N,
\]
it follows from \eqref{eq:amari_large_first_moment} that
\[
\mathbb P\bigl(\widehat T_i(\xi^1)\neq \xi_i^1,\ \xi_i^1=0,\ A_\delta\bigr) \le \left(\frac{(2-\delta)\alpha}{\gamma}\right)^{n-1}.
\]
Recalling that $n=\kappa\log N$, we obtain
\[
\left(\frac{(2-\delta)\alpha}{\gamma}\right)^{n-1}
=
N^{\kappa \log\bigl(\frac{(2-\delta)\alpha}{\gamma}\bigr)+o(1)}.
\]

Choose $\alpha_0=\alpha_0(\kappa,\gamma)>0$ so small that
$\kappa\log\left(\frac{(2-\delta)\alpha_0}{\gamma}\right)<-2$,
then, also for every $0<\alpha<\alpha_0$,
\[
\kappa\log\left(\frac{(2-\delta)\alpha}{\gamma}\right)<-2.
\]
Since
\[
\left(\frac{(2-\delta)\alpha}{\gamma}\right)^{n-1}
= N^{\kappa\log\left(\frac{(2-\delta)\alpha}{\gamma}\right)+o(1)},
\]
we obtain, for all sufficiently large $N$,
\[
\mathbb P\bigl(\widehat T_i(\xi^1)\neq\xi_i^1,\ \xi_i^1=0,\ A_\delta\bigr) \leq N^{-3/2}.
\]

Finally, by a union bound over all coordinates,
\[
\mathbb P\bigl(\widehat T(\xi^1)\neq \xi^1,\ A_\delta\bigr)
\leq N\cdot N^{-3/2}
=N^{-1/2}\longrightarrow0.
\]
Since also $\mathbb P(A_\delta^c)\to0$, this proves
$\mathbb P\bigl(\widehat T(\xi^1)=\xi^1\bigr)\to 1$.
\end{proof}

We continue with the proof of Corollary \ref{cor:willshaw_large_n}:

\begin{proof}[Proof of Corollary \ref{cor:willshaw_large_n}]
We argue as in the proof of Theorem~\ref{theo:amari_large_n}. Let
$\delta\in(\kappa,1)$ be chosen there such that
$\gamma<\Gamma_\kappa(\delta)$ (recall the right hande side was defined in \eqref{eq:gamma_kappa_delta})
and consider the corresponding event $A_\delta$.

First let $i$ be such that $\xi_i^1=1$. Since the stored pattern $\xi^1$
itself generates every hyperedge consisting of $i$ and $n-1$ further active
coordinates of $\xi^1$, the contribution to the local field is at least $(r_N)_{n-1}$, where $r_N:=\lceil\delta\log N\rceil-1$.
The estimate \eqref{eq:falling_factorial_asymptotics} from the proof of
Theorem~\ref{theo:amari_large_n} therefore implies that, for all sufficiently
large $N$,
$(r_N)_{n-1}>h$.
Hence every active coordinate is stable on $A_\delta$.

Now let $i$ be such that $\xi_i^1=0$. For every choice of pairwise distinct
indices $i_1,\ldots,i_{n-1}$ different from $i$, the corresponding Willshaw
weight satisfies
\[
\mathbf 1\left\{
\sum_{\mu=2}^M
\xi_i^\mu\xi_{i_1}^\mu\cdots\xi_{i_{n-1}}^\mu
\geq1
\right\}
\leq
\sum_{\mu=2}^M
\xi_i^\mu\xi_{i_1}^\mu\cdots\xi_{i_{n-1}}^\mu.
\]
Thus the false local field in the Willshaw model is bounded above by the
corresponding false local field in the Amari model. Consequently, the
false-positive event for an inactive coordinate in the Willshaw dynamics is
contained in the corresponding event for the Amari dynamics.

The bounds obtained in the proof of Theorem~\ref{theo:amari_large_n} therefore
apply without change. Together with the stability of the active coordinates on
$A_\delta$ and $\mathbb P(A_\delta^c)\to0$, this yields
\[
\mathbb P\bigl(\widehat T(\xi^1)=\xi^1\bigr)\longrightarrow1.
\]
\end{proof}

Let us finally turn to the GB model with growing interaction order $n=\kappa c$. The proof follows the same general strategy as in the fixed-order case. The main difference is that, since $n$ now grows with $c$,
the factor $ \binom{X_\nu}{n-1}$
is more sensitive to fluctuations of $X_\nu$. We therefore control
$X_\nu$ on a shrinking relative window of order $1/c$. Since
\[
\mathbb E[X_\nu]=\frac{l-1}{c}\sim\frac{l}{\log l}
\]
while $n=\kappa c=O(\log l)$, this window is still sufficiently large for Chernoff bounds to hold uniformly over all stored messages. On this event the false local field can again be reduced to a binomial counting variable.

\begin{proof}[Proof of Theorem \ref{thm:GB_logarithmic_n}]
By symmetry, it suffices to consider the stored message $\xi^1$. We condition
on $\xi^1$ throughout the proof. Since every message contains exactly one
active neuron in each block, the estimates below do not depend on the
particular realization of $\xi^1$.

For every block $a$, let $(a,k_a^1)$ denote the unique active neuron of
$\xi^1$ in that block.

For a correct neuron $(a,k_a^1)$, the message $\xi^1$ itself creates all hyperedges
\[
((a,k_a^1),(a_2,k_{a_2}),\dots,(a_n,k_{a_n}^1(1+o(1))e^{2\kappa\eta}c^{-(n-1)}K_n\,(1+\eta)\frac{M}{c}))
\]
with $a_2,\dots,a_n$ pairwise distinct and different from $a$. Hence
\[
S_{(a,k_a^1)}(\xi^1)\ge \binom{l-1}{n-1}>h,
\]
since $\gamma<1$. Therefore
$T_{(a,k_a^1)}(\xi^1)=1.$

Now fix a false neuron $(a,k)$, i.e.\ $\xi^1_{(a,k)}=0$. Define
\[
X_\nu:=\sum_{b\ne a}\mathbf 1\{\xi^\nu_{(b,k_b^1)}=1\},
\qquad
B_\nu:=\mathbf 1\{\xi^\nu_{(a,k)}=1\},
\qquad \text{as well as } \quad
Z_\nu:=B_\nu\binom{X_\nu}{n-1},
\]
for $\nu=2,\dots,M$. Then
\[
S_{(a,k)}(\xi^1)\le \sum_{\nu=2}^M Z_\nu .
\]
Indeed, every active hyperedge contributing to $S_{(a,k)}(\xi^1)$ is generated by
at least one stored message $\xi^\nu$, $\nu\ge2$, counted by $Z_\nu$; since the
same hyperedge may be generated by several messages, this gives only an upper bound.

Set
\[
r:=n-1,\qquad \lambda:=\frac{l-1}{c},
\qquad, \text{and}\quad K_n:=\binom{l-1}{r}.
\]
Since in each block $b\ne a$ the message $\xi^\nu$ hits the prescribed neuron
$k_b^1$ with probability $1/c$, independently across blocks, we have
\[
X_\nu\sim \mathrm{Bin}(l-1,1/c),
\]
hence $\mathbb E[X_\nu]=\lambda$.

Next, let us fix $\eta>0$ and define
$ \varepsilon_c:=\frac{\eta}{c}$
and consider the event
\[
E_\nu:=\{|X_\nu-\lambda|\le \varepsilon_c\lambda\}.
\]
By a standard Hoeffding-Chernoff bound for binomial random variables, 
since $X_\nu\sim\mathrm{Bin}(l-1,1/c)$ has mean $\lambda=(l-1)/c$ and
\[
E_\nu^c=\{|X_\nu-\lambda|>\varepsilon_c\lambda\}
\qquad\text{with}\qquad
\varepsilon_c=\eta/c,
\]
we obtain
\[
\mathbb P(E_\nu^c)\le 2\exp(-C\varepsilon_c^2\lambda)
=2\exp\!\left(-C_\eta\frac{\lambda}{c^2}\right).
\]
for some  $C_{\eta}>0$.

Since $c\sim\log l$ and $\lambda\sim l/\log l$, this yields
\be \label{eq:estimate_E_c}
\mathbb P(E_\nu^c)\le 2\exp\!\left(-C'_{\eta}\frac{l}{(\log l)^3}\right).
\ee
Consequently,
\be \label{eq:estimate_E_c2}
M\,\mathbb P(E_\nu^c)\to0,
\ee
because
\[
\log M=\log \alpha + n\log c = O(c\log c)=O(\log l\,\log\log l).
\]
Hence, with probability tending to one, all events $E_\nu$ occur simultaneously.

On the event $E_\nu$, we have $X_\nu\leq(1+\varepsilon_c)\lambda$.
Therefore,
\[
\binom{X_\nu}{r}\leq\frac{X_\nu^r}{r!}\leq
\frac{(1+\varepsilon_c)^r\lambda^r}{r!}.
\]
Since $\lambda=(l-1)/c$ and
$K_n=\frac{(l-1)_r}{r!}$, we obtain
\[
\binom{X_\nu}{r} \leq (1+\varepsilon_c)^r\frac{(l-1)^r}{(l-1)_r}\,
c^{-r}K_n.
\]
Now $r=\kappa c+O(1)$ and $\varepsilon_c=\eta/c$, so
\[
(1+\varepsilon_c)^r
=
\exp\bigl(\kappa\eta+o(1)\bigr).
\]
Moreover, since $r=O(\log l)$,
\[
\log\frac{(l-1)^r}{(l-1)_r}
=
-\sum_{j=0}^{r-1}\log\left(1-\frac{j}{l-1}\right)
=
O\left(\frac{r^2}{l}\right)
=o(1).
\]
Hence, on $E_\nu$,
\[
\binom{X_\nu}{r}
\leq
\exp\bigl(\kappa\eta+o(1)\bigr)c^{-r}K_n.
\]

Therefore, on the event $\bigcap_{\nu=2}^M E_\nu$,
\[
\sum_{\nu=2}^M Z_\nu
\le
e^{\kappa\eta+o(1)}c^{-(n-1)}K_n\sum_{\nu=2}^M B_\nu.
\]
Since the variables $B_\nu$ are i.i.d.\ Bernoulli$(1/c)$, we have
\[
\sum_{\nu=2}^M B_\nu\sim \mathrm{Bin}(M-1,1/c),
\]
with mean
\[
\mu=\frac{M-1}{c}\sim \frac{M}{c}.
\]
Hence, by a standard Chernoff bound, there exists $C_\eta>0$ such that
\[
\mathbb P\left(\sum_{\nu=2}^M B_\nu\ge (1+\eta)\frac{M}{c}\right)
\le
\exp\left(-C_\eta \frac{M}{c}\right)
\]
for all sufficiently large $l$. On this event, and on $\bigcap_{\nu=2}^M E_\nu$, we obtain
\[
\sum_{\nu=2}^M Z_\nu
\leq e^{\kappa\eta+o(1)} c^{-(n-1)}K_n\,(1+\eta)\frac{M}{c}.
\]
Because $M\leq\alpha c^n=\alpha c^{r+1}$, this gives
\[
\sum_{\nu=2}^M Z_\nu
\leq (1+\eta)e^{\kappa\eta+o(1)}\alpha K_n.
\]

Since $\alpha<\gamma$, we may choose $\eta>0$ so small that
$(1+\eta)e^{\kappa\eta}\alpha<\gamma$.
Hence, for all $l$ large enough,
\[
S_{(a,k)}(\xi^1)<\gamma K_n=h
\]
with probability tending to one. 
Hence, conditionally on $\xi^1$, for every false neuron $(a,k)$,
\[
\mathbb P\bigl(T_{(a,k)}(\xi^1)=1\,\big|\,\xi^1\bigr)
\leq
2M\exp\left(-C'_\eta\frac{l}{(\log l)^3}\right)
+
\exp\left(-C_\eta\frac{M}{c}\right).
\]
For the first term, since
\[
\log M=O(\log l\,\log\log l),
\]
we have
\[
2M\exp\left(-C'_\eta\frac{l}{(\log l)^3}\right)
=o((lc)^{-1}).
\]
For the second term, observe that
$\frac{M}{c}=(\alpha+o(1))c^{n-1}$
and, since $n=\kappa c$ and $c\sim\log l$,
\[
c^{n-1}
=\exp\bigl((\kappa c+O(1))\log c\bigr)
\]
grows faster than any power of $l$. Hence
\[
\exp\left(-C_\eta\frac{M}{c}\right)=o((lc)^{-1}).
\]
Thus, uniformly in the realization of $\xi^1$ and in the choice of
the false neuron,
\[
\mathbb P\bigl(T_{(a,k)}(\xi^1)=1\,\big|\,\xi^1\bigr)
=o((lc)^{-1}).
\]
Every realization of $\xi^1$ contains exactly $l(c-1)$ false neurons, while
all its active neurons are stable deterministically. Therefore, by a union
bound,
\[
\mathbb P\bigl(T(\xi^1)\neq\xi^1\,\big|\,\xi^1\bigr)
\leq l(c-1)o((lc)^{-1}) =o(1),
\]
uniformly in $\xi^1$. Averaging over $\xi^1$ yields
\[
\mathbb P\bigl(T(\xi^1)\neq\xi^1\bigr)\longrightarrow0.
\]
This proves the theorem.

\end{proof}

\begin{rem}
The logarithmic-order result for the GB model should be compared with the
corresponding higher-order sparse model with some care. In the GB architecture,
messages are constrained to have exactly one active neuron per block, and therefore
a fixed $n$-tuple of neurons from distinct blocks is realized by a given message
with probability $c^{-n}$. This leads to the natural storage scale $M\asymp c^n$.

In the unrestricted sparse model, on the other hand, the class of admissible patterns
is much larger, and this permits considerably larger storage capacities. Thus the
difference in scale is a structural feature of the models rather than an artefact of
the proof.
\end{rem}

\begin{rem}
Theorem~\ref{thm:GB_logarithmic_n} is a fixed-point stability result for the
retrieval dynamics. It shows that a fixed stored GB-message is, with probability
tending to one, a fixed point of the retrieval dynamics. In particular, the
result does not by itself imply the existence of a non-trivial basin of attraction
as in \cite{DHLUV17}.

We do not address here the correction of partially corrupted GB-messages.
Indeed, if $r$ blocks of a stored message are corrupted, then the deterministic
signal term is reduced from $\binom{l-1}{n-1}$ to approximately
\[
\binom{l-1-r}{n-1},
\]
while the corresponding noise analysis becomes substantially more involved.
It would therefore be natural to investigate one-step or multi-step recovery of
corrupted GB-messages in future work.
\end{rem}

\vspace{2mm}
{\bf Data Availability Statement.}
No datasets were generated or analyzed during the current study.

\vspace{2mm}
{\bf Conflict of interest.} The authors declare no competing interests. 

\vspace{2mm}
{\bf Acknowledgements.}
Matthias L\"owe was supported by the German Research Foundation under Germany's Excellence Strategy EXC 2044/2 - 390685587, Mathematics M\"unster: Dynamics - Geometry - Structure. 

\newcommand{\etalchar}[1]{$^{#1}$}


\begin{thebibliography}{MPRV87}
	
	\bibitem[AABS25]{albanese2025exponentialhopfieldmodel}
	Linda Albanese, Andrea Alessandrelli, Adriano Barra, and Peter Sollich.
	\newblock Yet another exponential {H}opfield model.
	\newblock {\em arXiv preprint arXiv:2509.06905}, 2025.
	
	\bibitem[ABGJ14]{griponc}
	Behrooz~Kamary Aliabadi, Claude Berrou, Vincent Gripon, and Xiaoran Jiang.
	\newblock Storing sparse messages in networks of neural cliques.
	\newblock {\em IEEE Transactions on Neural Networks and Learning Systems},
	25:980--989, 2014.
	
	\bibitem[AFM20]{Agliari_2020}
	Elena Agliari, Alberto Fachechi, and Chiara Marullo.
	\newblock The relativistic {H}opfield model with correlated patterns.
	\newblock {\em Journal of Mathematical Physics}, 61(12), December 2020.
	
	\bibitem[Amb24]{ambrogioni2024search}
	Luca Ambrogioni.
	\newblock In search of dispersed memories: Generative diffusion models are
	associative memory networks.
	\newblock {\em Entropy}, 26(5):381, 2024.
	
	\bibitem[Bur94]{burshtein}
	David Burshtein.
	\newblock Nondirect convergence radius and number of iterations of the
	{H}opfield associative memory.
	\newblock {\em IEEE Trans. Inform. Theory}, 40(3):838--847, 1994.
	
	\bibitem[DHL{\etalchar{+}}17]{DHLUV17}
	Mete Demircigil, Judith Heusel, Matthias L\"{o}we, Sven Upgang, and Franck
	Vermet.
	\newblock On a model of associative memory with huge storage capacity.
	\newblock {\em J. Stat. Phys.}, 168(2):288--299, 2017.
	
	\bibitem[GHLV16]{GHLV16}
	Vincent Gripon, Judith Heusel, Matthias L\"{o}we, and Franck Vermet.
	\newblock A comparative study of sparse associative memories.
	\newblock {\em J. Stat. Phys.}, 164(1):105--129, 2016.
	
	\bibitem[Hoe63]{Hoeffding1963}
	Wassily Hoeffding.
	\newblock Probability inequalities for sums of bounded random variables.
	\newblock {\em J. Amer. Statist. Assoc.}, 58:13--30, 1963.
	
	\bibitem[Hop82]{Hopfield1982}
	J.~J. Hopfield.
	\newblock Neural networks and physical systems with emergent collective
	computational abilities.
	\newblock {\em Proc. Nat. Acad. Sci. U.S.A.}, 79(8):2554--2558, 1982.
	
	\bibitem[iA89]{Amari1989}
	Shun ichi Amari.
	\newblock Characteristics of sparsely encoded associative memory.
	\newblock {\em Neural Networks}, 2(6):451 -- 457, 1989.
	
	\bibitem[KH16]{KrotovHopfield2016}
	Dmitry Krotov and John~J. Hopfield.
	\newblock Dense associative memory for pattern recognition.
	\newblock In {\em Proceedings of the 30th International Conference on Neural
		Information Processing Systems}, NIPS'16, pages 1180--1188, Red Hook, NY,
	USA, 2016. Curran Associates Inc.
	
	\bibitem[KH21]{krotov2021largeassociativememoryproblem}
	Dmitry Krotov and John Hopfield.
	\newblock Large associative memory problem in neurobiology and machine
	learning.
	\newblock {\em arXiv preprint arXiv:2008.06996}, 2021.
	
	\bibitem[Lou97]{loukianova}
	Daria Loukianova.
	\newblock Lower bounds on the restitution error in the {H}opfield model.
	\newblock {\em Probab. Theory Related Fields}, 107(2):161--176, 1997.
	
	\bibitem[L{\"o}w98]{Lo98}
	Matthias L{\"o}we.
	\newblock On the storage capacity of {H}opfield models with correlated
	patterns.
	\newblock {\em Ann. Appl. Probab.}, 8(4):1216--1250, 1998.
	
	\bibitem[LV05a]{LV_BEG}
	Matthias L{\"o}we and Franck Vermet.
	\newblock The storage capacity of the {B}lume-{E}mery-{G}riffiths neural
	network.
	\newblock {\em J. Phys. A}, 38(16):3483--3503, 2005.
	
	\bibitem[LV05b]{LV05}
	Matthias L{\"o}we and Franck Vermet.
	\newblock The storage capacity of the {H}opfield model and moderate deviations.
	\newblock {\em Statist. Probab. Lett.}, 75(4):237--248, 2005.
	
	\bibitem[LV11]{LV11}
	Matthias L{\"o}we and Franck Vermet.
	\newblock The {H}opfield model on a sparse {E}rd{\H o}s-{R}enyi graph.
	\newblock {\em J. Stat. Phys.}, 143(1):205--214, 2011.
	
	\bibitem[LV15]{LV14}
	Matthias L{\"o}we and Franck Vermet.
	\newblock Capacity of an associative memory model on random graph
	architectures.
	\newblock {\em Bernoulli}, 21(3):1884--1910, 2015.
	
	\bibitem[MPRV87]{MPRV}
	Robert~J. McEliece, Edward~C. Posner, Eugene~R. Rodemich, and Santosh~S.
	Venkatesh.
	\newblock The capacity of the {H}opfield associative memory.
	\newblock {\em IEEE Trans. Inform. Theory}, 33(4):461--482, 1987.
	
	\bibitem[New88]{Newman_hopfield}
	Charles~M. Newman.
	\newblock Memory capacity in neural network models: Rigorous lower bounds.
	\newblock {\em Neural Networks}, 1(3):223--238, 1988.
	
	\bibitem[Oka96]{Okada1996}
	Masato Okada.
	\newblock Notions of associative memory and sparse coding.
	\newblock {\em Neural Networks}, 9(8):1429 -- 1458, 1996.
	\newblock Four Major Hypotheses in Neuroscience.
	
	\bibitem[PRN{\etalchar{+}}25]{pham2025memorization}
	Bao Pham, Gabriel Raya, Matteo Negri, Mohammed~J Zaki, Luca Ambrogioni, and
	Dmitry Krotov.
	\newblock Memorization to generalization: Emergence of diffusion models from
	associative memory.
	\newblock {\em arXiv preprint arXiv:2505.21777}, 2025.
	
	\bibitem[RSL{\etalchar{+}}20]{ramsauer2020hopfield}
	Hubert Ramsauer, Bernhard Sch{\"a}fl, Johannes Lehner, Philipp Seidl, Michael
	Widrich, Thomas Adler, Lukas Gruber, Markus Holzleitner, Milena Pavlovi{\'c},
	Geir~Kjetil Sandve, et~al.
	\newblock Hopfield networks is all you need.
	\newblock {\em arXiv preprint arXiv:2008.02217}, 2020.
	
	\bibitem[Tal98]{Talagrand_hopfield}
	Michel Talagrand.
	\newblock Rigorous results for the {H}opfield model with many patterns.
	\newblock {\em Probab. Theory Related Fields}, 110(2):177--276, 1998.
	
	\bibitem[WBL69]{Willshaw}
	D.~J. {Willshaw}, O.~P. {Buneman}, and H.~C. {Longuet-Higgins}.
	\newblock {Non-Holographic Associative Memory}.
	\newblock {\em Nature}, 222:960--962, June 1969.
	
\end{thebibliography}
\end{document}